\documentclass{amsart}
\usepackage[all]{xy}
\usepackage{amssymb}
\usepackage{stmaryrd}
\usepackage{url}

\newcommand{\arrows}{\rightrightarrows}
\newcommand{\id}{\mathrm{id}}
\newcommand{\cD}{\mathcal{D}}
\newcommand{\cbrack}[2]{\llbracket #1, #2 \rrbracket}
\DeclareMathOperator{\Ann}{Ann}

\newtheorem{thm}{Theorem}[section]
\newtheorem{prop}[thm]{Proposition}
\newtheorem{lemma}[thm]{Lemma}
\newtheorem{cor}[thm]{Corollary}

\theoremstyle{definition}
\newtheorem{definition}[thm]{Definition}

\theoremstyle{remark}
\newtheorem{remark}[thm]{Remark}
\newtheorem{remarks}[thm]{Remarks}

\newtheorem{example}[thm]{Example}

\numberwithin{equation}{section}

\begin{document}

\title
{Constant Symplectic $2$-groupoids}
\author{Rajan Amit Mehta}
\address{Department of Mathematics \& Statistics\\
Smith College\\
44 College Lane\\
Northampton, MA 01063}
\email{rmehta@smith.edu}

\author{Xiang Tang}
\address{Department of Mathematics\\
Washington University in Saint Louis\\
One Brookings Drive\\
Saint Louis, Missouri, USA 63130}
\email{xtang@math.wustl.edu}

\begin{abstract}
We propose a definition of symplectic $2$-groupoid which includes integrations of Courant algebroids that have been recently constructed. We study in detail the simple but illustrative case of constant symplectic $2$-groupoids. We show that the constant symplectic $2$-groupoids are, up to equivalence, in one-to-one correspondence with a simple class of Courant algebroids that we call constant Courant algebroids. Furthermore, we find a correspondence between certain Dirac structures and Lagrangian sub-$2$-groupoids.
\end{abstract}

\subjclass[2010]{53D17, 58H05} 
\keywords{Dirac, symplectic, $2$-groupoids, Courant algebroids}

\maketitle

\section{Introduction}
In \cite{lwx:courant}, Liu, Weinstein, and Xu introduced the notion of a \emph{Courant algebroid}, axiomatizing the brackets studied by Courant and Weinstein \cite{cw:courant, courant:dirac} and Dorfman \cite{dorfman}. Recognizing the similarities and relationships between Lie algebroids and Courant algebroids, they asked whether there is a groupoid-like object that can be viewed as the ``integration'' of a Courant algebroid. 

In \cite{severa:some}, \v{S}evera outlined a construction by which, given a Courant algebroid, one could formally produce a $2$-groupoid as a moduli space of maps of dg-manifolds. The space of $2$-simplices of this $2$-groupoid possesses a symplectic form, suggesting that the integration of a Courant algebroid should be called a \emph{symplectic $2$-groupoid}. In the case of an exact Courant algebroid, the symplectic $2$-groupoid arising from \v{S}evera's construction (the \emph{Liu-Weinstein-Xu $2$-groupoid}) was explicitly described and shown to be smooth in \cite{mt:exact}.

In all but the most trivial cases, the $2$-groupoids arising from \v{S}evera's construction are infinite-dimensional, so it is reasonable to look for finite-dimensional models. In different (but overlapping) special cases, such models were independently found in \cite{ls:integration, mt:double, sz:integration}, each with its own shortcoming. In \cite{sz:integration}, there is no $2$-form constructed on the integration. In \cite{ls:integration}, the integration has a symplectic structure but is only defined locally\footnote{Recently, \v{Severa} and \v{S}ira\v{n} \cite{severa-siran} described a general construction for local integration and showed that the local integrations can be glued together up to coherent homotopy.}. The construction in \cite{mt:double} is global, but the $2$-form constructed there is degenerate.

Based on the class of examples discovered there, the authors in \cite{mt:double} suggested a definition of symplectic $2$-groupoid where the $2$-form is allowed to be degenerate, but where the degeneracy is controlled in a certain way by the simplicial structure. However, the definition given there fails to serve as a good general definition. Specifically, \cite[Definition 6.7]{mt:double} implies that when the $2$-form is genuinely nondegenerate at a point $x$ of the unit space $X_0$, the  tangent space of $X_1$ at $x$ is isomorphic to the tangent space of $X_0$ at $x$, which forces $X_1$ to have the same dimension of $X_0$.  This excludes too many interesting examples, including the case of constant symplectic $2$-groupoids to be presented in this paper. This observation suggests that the definition requires revision. 

In Definition \ref{dfn:symplectic2groupoid}, we give a new definition of symplectic $2$-groupoid which still includes the examples constructed in \cite{mt:double} (see Section \ref{sec:double}). This definition is partly inspired by the notion of a shifted symplectic structure \cite{ptvv:derived}, and we are particularly indebted to Getzler \cite{getzler:slides}, who recast many of those ideas in the concrete language of simplicial manifolds. However, we stress that our definition is in several ways more strict than the one coming from \cite{ptvv:derived}. Specifically, our definition only includes a $2$-form on the space of $2$-simplices $X_2$, and we require this $2$-form to be closed on the nose. Additionally, our nondegeneracy requirement is more strict. 

The strictness of our definition reflects the fact that the solution to the integration problem will not be Morita invariant. For example, the notion of symplectic groupoid \cite{cdw} is not Morita invariant, and this strictness is necessary in order to obtain a Lie-theoretic correspondence with Poisson manifolds. On the other hand, the notion of $1$-shifted symplectic structure agrees with Xu's \cite{xu-quasi} weaker notion of \emph{quasi-symplectic groupoid}.

In the second half of this article, we consider the case of \emph{constant symplectic $2$-groupoids}, i.e.\ symplectic $2$-groupoids with a linear structure with respect to which the $2$-form is constant. Besides being a basic test case for the definition, constant symplectic $2$-groupoids should be useful for understanding the general case since they appear as first-order approximations of arbitrary symplectic $2$-groupoids. Specifically, given a symplectic $2$-groupoid $(X_\bullet, \omega)$ and a point $x \in X_0$, the ``tangent space at $x$'' $(T|_x X_\bullet, \omega_x)$ is a constant symplectic $2$-groupoid. In this sense, constant symplectic $2$-groupoids play the same role in the study of symplectic $2$-groupoids as symplectic vector spaces play in the study of symplectic manifolds.

We find that there is a relationship between constant symplectic $2$-groupoids and a certain class of Courant algebroids that we call \emph{constant Courant algebroids}. Specifically, the main results are as follows: 
\begin{enumerate}
\item (Theorem \ref{thm:correspondence}) There is a one-to-one correspondence between constant Courant algebroids and equivalence classes of constant symplectic $2$-groupoids.
\item  (Theorem \ref{thm:diracintegrate}) Under the above correspondence, constant Dirac structures are in one-to-one correspondence with wide linear Lagrangian sub-$2$-groupoids. 
\end{enumerate}
In other words, constant Courant algebroids integrate to constant symplectic $2$-groupoids, and constant Dirac structures integrate to certain Lagrangian sub-$2$-groupoids. These results provide evidence in support of our definition of symplectic $2$-groupoids as being the correct answer to the question posed by Liu, Weinstein, and Xu.

\subsection*{Organization of the paper}
In Section \ref{sec:symplectic2groupoid}, we define symplectic $2$-groupoids and show that the class of examples from \cite{mt:double} satisfies the definition. In Section \ref{sec:courant}, we study constant symplectic $2$-groupoids and find a minimal description of them in terms of linear algebra data.
In Section \ref{sec:constcourant}, we similarly study constant Courant algebroids and show that they are in correspondence with constant symplectic $2$-groupoids. Finally, in Section \ref{sec:dirac}, we consider linear Lagrangian sub-$2$-groupoids and describe the correspondence with constant Dirac structures.

\subsection*{Acknowledgments} 
We would like to thank Ezra Getzler for inspiring discussions and explanations about symplectic structures on differentiable $n$-stacks. We would like to thank  Damien Calaque for a discussion about the relation between integration of Courant algebroids and derived symplectic geometry. The research of the second author is partially supported by NSF grant DMS 1363250. 

\section{Symplectic $2$-groupoids}\label{sec:symplectic2groupoid}

\subsection{Lie $n$-groupoids and differential forms}
We start by recalling the definition of a Lie $n$-groupoid (see \cite{duskin, H, Z}).

\begin{definition} \label{dfn:simplicial}A \emph{simplicial manifold} is a sequence $X_\bullet = \{X_q\}$, $q \geq 0$, of manifolds equipped with surjective submersions $f_i^q: X_q \to X_{q-1}$ (called \emph{face maps}), $i=0,\dots,q$,  and embeddings $\sigma_i^q : X_q \to X_{q+1}$ (called \emph{degeneracy maps}), $i=0, \dots, q$,  such that
\begin{align}
f_i^{q-1}f_j^q &= f_{j-1}^{q-1}f_i^q, &i < j,\label{eqn:twoface}\\
\sigma_i^{q+1}\sigma_j^q &= \sigma_{j+1}^{q+1}\sigma_i^q, &i < j, \label{eqn:twodegen}\\
f_i^{q+1}\sigma_j^q &= \begin{cases}
\sigma_{j-1}^{q-1}f_i^q, & i< j,\\
\id, & i = j,j+1, \\
\sigma_j^{q-1}f_{i-1}^q, & i > j+1. \label{eqn:facedeg}\end{cases}
\end{align}
\end{definition}

For $q \geq 1$ and $0 \leq k \leq q$, recall that a $(q,k)$-\emph{horn} of $X_\bullet$ consists of a $q$-tuple $(x_0, \dots, x_{k-1}, x_{k+1}, \dots x_q)$, where $x_i \in X_{q-1}$, satisfying the \emph{horn compatibility equations}
\begin{equation}\label{eqn:horncompat}
     f_i^{q-1} x_j = f_{j-1}^{q-1} x_i
\end{equation}
for $i < j$.  The space of all $(q,k)$-horns is denoted $\Lambda_{q,k}X$.

The natural \emph{horn maps} $\lambda_{q,k}: X_q \to \Lambda_{q,k} X$ are defined as 
\[
 \lambda_{q,k}(x) = (f_0^q x, \dots, \widehat{f_k^q x}, \dots f_q^q x)
\]
 for $x \in X_q$. It is immediate from \eqref{eqn:twoface} that $\lambda_{q,k}(x)$ satisfies the horn compatibility equations \eqref{eqn:horncompat}; in fact, the purpose of the horn compatibility equations is to axiomatize the properties satisfied by $\lambda_{q,k}(x)$.

\begin{definition}\label{dfn:n}
     A \emph{Lie $n$-groupoid} is a simplicial manifold such that the horn maps $\lambda_{q,k}$ are
\begin{enumerate}
     \item surjective submersions for all $q \geq 1$, and
     \item diffeomorphisms for all $q > n$.
\end{enumerate}
\end{definition}

Given a simplicial manifold $X_\bullet$, we consider the bigraded space of differential forms $\Omega^\bullet(X_\bullet)$. There are two natural commuting differentials on $\Omega^\bullet(X_\bullet)$. One is the de Rham differential $d: \Omega^p(X_q)\to \Omega^{p+1}(X_q)$, and the other is the simplicial coboundary operator $\delta: \Omega^p(X_q)\to \Omega^p(X_{q+1})$, given by $\delta \alpha = \sum_{i=0}^{q+1} (-1)^i (f_i^{q+1})^* \alpha$. 

\begin{definition}\label{dfn:multandnormal}
A form $\alpha\in \Omega^p(X_q)$ is {\em multiplicative} if $\delta \alpha=0$. A form $\alpha \in \Omega^p(X_q)$ is {\em normalized} if $(\sigma_{q-1}^i)^*\alpha=0$ for all $i=0, \dots, q-1$.
\end{definition}
We note that the normalization condition holds vacuously in the case $q=0$. 

The space of normalized forms is denoted $\Omega^\bullet_\nu(X_\bullet)$. It is not hard to check that $\Omega^\bullet_\nu(X_\bullet)$ is closed under $d$ and $\delta$.

\subsection{The tangent complex}\label{subsec:tangent}

Let $X_\bullet$ be a simplicial manifold. For each $q > 0$, let $\sigma^q: X_0 \to X_q$ be defined as $\sigma^q = \sigma^{q-1}_0\cdots \sigma^0_0$. We can think of the image of $\sigma^q$ as being the ``unit space'' in $X_q$.

For $x\in X_0$ let $T_{x,q}X$ denote the tangent space $T_{\sigma^q(x)}X_q$. There is a natural simplicial structure on $T_{x,\bullet}X$ where the face and degeneracy maps are restrictions of the differentials of $f_i^q$ and $\sigma_i^q$. We note that $T_{x,\bullet}X$ is a simplicial vector space, in the sense that each $T_{x,q}X$ is a vector space and all of the face and degeneracy maps are linear.

There is a natural boundary map $\partial_q: T_{x,q}X \to T_{x,q-1}X$, given by
\[ \partial_q(v) = \sum_{i=0}^q (-1)^i (f_i^q)_* v.\]
It follows from \eqref{eqn:twoface} that $\partial^2 = 0$.

The Dold-Kan correspondence, c.f. \cite{go-ja:book}, associates a chain complex $\hat{T}_{x,\bullet}X$ to the simplicial vector space $T_{x, \bullet}X$. This chain complex can be explicitly described as follows. The \emph{normalized tangent space} $\hat{T}_{x,q}X$ is defined to be the quotient of $T_{x,q}X$ by the sum of the degenerate subspaces:
\[
\hat{T}_{x,q}X := \left. T_{x,q}X \middle/ \sum_{i=0}^{q-1} (\sigma^{q-1}_i)_* T_{x,q-1} X \right. .
\]
In particular, $\hat{T}_{x,0} = T_{x,0}$.

The boundary map $\partial$ descends to the normalized tangent spaces, so we have a chain complex
\[
\cdots\rightarrow \hat{T}_{x,q}X\xrightarrow{\partial} \hat{T}_{x,q-1}X \xrightarrow{\partial} \cdots \xrightarrow{\partial} \hat{T}_{x,1}X \xrightarrow{\partial} \hat{T}_{x,0}X,
\]
which is called the \emph{tangent complex} of the simplicial manifold $X_\bullet$ at $x\in X_0$. Taken together, the tangent complexes at every $x \in X_0$ form a complex of vector bundles over $X_0$, called the tangent complex of $X_\bullet$.

\begin{prop} \label{prop:n-complex}
If $X_\bullet$ is a Lie $n$-groupoid, then $\hat{T}_{x,q}$ is trivial for $q > n$. 
\end{prop}
In this paper, we will only make use of Proposition \ref{prop:n-complex} in the case $n=2$. A proof in this case is essentially contained in Section \ref{subsec:simplicial}; in particular, see Remark \ref{rmk:wtangent}. We leave the general case to the reader.

\begin{example}
	In the case where $X_\bullet$ is the nerve of a Lie groupoid $G \arrows M$, then the tangent complex can be identified with the $2$-term complex $A \xrightarrow{\rho} TM$, where $A$ is the Lie algebroid of $G$ and $\rho$ is the anchor map.
\end{example}

\subsection{Forms on Lie $2$-groupoids}\label{subsec:lie2}
We will now restrict our attention to the case of Lie $2$-groupoids. From Proposition \ref{prop:n-complex}, we know that the tangent complex of a Lie $2$-groupoid $X_\bullet$ is a $3$-term complex
\[
\hat{T}_2 X \xrightarrow{\partial} \hat{T}_1 X \xrightarrow{\partial} T_0 X.
\]
Given a normalized $2$-form $\omega \in \Omega^2_\nu (X_2)$, we can obtain the following bilinear pairings on the tangent complex at any $x \in X_0$ (we learned of these from Getzler \cite{getzler:slides}): 
\begin{enumerate}
	\item For $v\in T_{x,0}X$ and $w\in T_{x,2}X$, let 
	\begin{equation}\label{eqn:tildeA}
\widetilde{A}_\omega( v, w):=\omega(\sigma^2_*v, w).
\end{equation}
It follows from the assumption that $\omega$ is normalized that $\widetilde{A}_\omega$ descends to a well-defined bilinear pairing $A_\omega$ between $T_{x,0}X$ and $\hat{T}_{x,2}X$. 
\item For $\theta, \eta\in T_{x,1}X$, let
\begin{equation}\label{eqn:tildeB}
\widetilde{B}_\omega(\theta, \eta):=\omega\left((\sigma^1_1)_*\theta, (\sigma^1_0)_*\eta\right)
+\omega\left((\sigma^1_1)_*\eta, (\sigma^1_0)_*\theta \right). 
\end{equation}
It follows from the assumption that $\omega$ is normalized that $\widetilde{B}_\omega$ descends to a well-defined symmetric bilinear form $B_\omega$ on $\hat{T}_{x,1}X$. 
\end{enumerate}

\begin{definition}\label{dfn:nondegeneracy}
	A normalized $2$-form $\omega \in \Omega^2_\nu(X_2)$ is called \emph{simplicially nondegenerate} if the induced pairings $A_\omega$ and $B_\omega$ are nondegenerate at all $x \in X_0$.
\end{definition}

\begin{definition}\label{dfn:symplectic2groupoid}A \emph{symplectic $2$-groupoid} is a Lie $2$-groupoid $X_\bullet$ equipped with a closed, multiplicative, normalized, and simplicially nondegenerate $2$-form $\omega \in \Omega^{2}_\nu(X_{2})$.
\end{definition}

\begin{definition}\label{dfn:equivalence} Two closed, multiplicative, normalized, and simplicially nondegenerate $2$-forms $\omega, \omega' \in \Omega^{2}_\nu(X_2)$ are \emph{equivalent} if there exists a closed and normalized $2$-form $\alpha \in \Omega^2_\nu(X_1)$ such that $\omega' - \omega = \delta \alpha$ and $A_{\delta \alpha}=B_{\delta \alpha} = 0$.
\end{definition}

\begin{remark}\label{rmk:equivalence}
By a straightforward calculation using \eqref{eqn:tildeA} and \eqref{eqn:facedeg}, one can see that $A_{\delta \alpha} = 0$ if and only if
\begin{equation} \label{eqn:Adelta}
\alpha\left((\sigma_0^0)_*v, \partial w\right) = 0
\end{equation}
for all $v \in T_{x,0} X$ and $w \in T_{x,2} X$. Similarly, $B_{\delta \alpha} = 0$ if and only if
\begin{equation} \label{eqn:Bdelta}
\alpha\left(\theta, (\sigma_0^0)_* \partial \eta\right) + \alpha\left(\eta, (\sigma_0^0)_* \partial \theta\right) = 0
\end{equation}
for all $\theta, \eta \in T_{x,1} X$.
\end{remark}

\begin{remark}\label{rmk:normalized} It is known \cite{cdw} that any multiplicative $2$-form on a Lie groupoid is automatically normalized; thus the normalization condition does not explicitly appear in the definition of symplectic groupoid. However, in the case of Lie $2$-groupoids, normalization does not automatically follow from multiplicativity. For example, for any manifold $M$, let $X_k=M$ for all $k$, with all the face and degeneracy maps being the identity. It is immediate that any nonzero $2$-form $\omega$ on $X_2=M$  is multiplicative but not normalized.
\end{remark}

\subsection{Example: an integration of $A \oplus A^*$}\label{sec:double}
In \cite{lwx:courant}, Liu, Weinstein, and Xu constructed a Courant algebroid $A \oplus A^*$ associated to any Lie bialgebroid $(A,A^*)$. This construction leads to a large and important class of Courant algebroids. In \cite{mt:double}, the authors described a method of integrating Courant algebroids of the form $A\oplus A^*$ by first integrating (if possible) the Lie bialgebroid $(A,A^*)$ to a symplectic double Lie groupoid $D$ and then applying the bar functor to obtain a Lie $2$-groupoid $\overline{W}ND$ equipped with a closed $2$-form.  The fact that this $2$-form is degenerate was the first clue that the correct notion of symplectic $2$-groupoid should allow for $2$-forms that have some degeneracy.

We will now prove that the integration of the standard Courant algebroid\footnote{The extension of the proof to the general case $\overline{W}ND$, where $D$ is a symplectic double Lie groupoid, is similar.} $TM \oplus T^*M$ in \cite{mt:double} satisfies Definition \ref{dfn:symplectic2groupoid}.   

We first recall the Lie $2$-groupoid obtained in \cite{mt:double} as integration of $TM \oplus T^*M$. For any $k \geq 1$, let $T^*_{(k)} M$ denote the direct sum of $k$ copies of $T^*M$.
For each $q \geq 0$, define $X_q$ as
\[
X_q = M\times T^*M\times \cdots \times T^*_{(q)}M. 
\]
Let $\tau^q_i: T^*_{(q)}M\to T^*M$ be the projection onto the $i$-th component, let $p_q: T^*_{(q)}M \to M$ denote the bundle projection map, and let $\iota^q: M \to T^*_{(q)}M$ be the zero section map. We will omit the index $q$ in $p_q$ and $\iota^q$ when $q=1$. Finally, for $i=1,2$, let $\iota_i^2: T^*M \to T^*_{(2)}M$ be defined as
\begin{align*}
 \iota_1^2(\xi) &= \big(\xi, \iota(p(\xi))\big), & \iota_2^2(\xi) &= \big(\iota(p(\xi)), \xi\big).
\end{align*}

The face and degeneracy maps between $X_0 = M$ and $X_1 = M \times T^*M$ are defined as
\begin{align*}
\sigma_0^0(x) &= (x, \iota(x)), &
f_0^1(x,\xi) &= x, & f_1^1(x, \xi) &= p(\xi),
\end{align*}
for $x \in M$ and $\xi \in T^*M$. The degeneracy maps $\sigma_i^1:X_1 \to X_2$ are defined as
\begin{align*}
\sigma^1_0(x, \xi) &=(x, \xi, \iota_1^2(\xi)), & \sigma^1_1(x, \xi) &= (x,\iota(x), \iota^2_2(\xi)), 
\end{align*}
and the face maps $f^2_i: X_2\to X_1$ are defined as
\begin{align*}
f^2_0(x, \xi, \xi^2)&:=(x, \xi),\\
f^2_1(x, \xi, \xi^2)&:=(x, \tau_1^2(\xi^2)+\tau_2^2(\xi^2)),\\
f^2_2(x, \xi, \xi^2)&:=(p(\xi), \tau_2^2(\xi^2)),
\end{align*}
for $x \in M$, $\xi \in T^*M$, and $\xi^2 \in T^*_{(2)}M$.

As we will only need the structure maps up to level $2$, we stop here and refer the reader to \cite{mt:double} for the general definitions of the simplicial structure maps.

There is a natural map $d: X_2=M\times T^*M\times T^*_{(2)} M\to T^*M\times T^*M$ given by 
\[
d(m, \xi, \xi^2):=(\xi, \tau^2_1(\xi^2)). 
\]
Let $\omega_0$ be the canonical symplectic form on $T^*M$. Then the $2$-form $\omega \in \Omega^2(X_2)$ is defined as the pullback by $d$ of $(\omega_0, -\omega_0)$. It is clear that $\omega$ is closed, we proved in \cite[Proposition 6.2]{mt:double} that $\omega$ is multiplicative, and it can be easily checked that $\omega$ is normalized.

\begin{prop}\label{prop:symplectic} The 2-form $\omega$ on $X_2=M\times T^*M\times T^*_{(2)}M$ is simplicially nondegenerate. Therefore, $(X_\bullet, \omega)$ is a symplectic $2$-groupoid.
\end{prop}
\begin{proof}
We start by describing the tangent spaces $T_{x,q}X$.  We observe that $T_{x,0} X = T_x M$, that $T_{x,1}X $ can be identified with $T_x M \oplus T_x M \oplus T^*_xM$, and that $T_{x,2}X$ can be similarly identified with $T_x M\oplus T_x M\oplus T_x^* M\oplus T_x M\oplus T^*_x M\oplus T^*_x M$. Using these identifications, we can describe the degeneracy maps $(\sigma_0^0)_*: T_{x,0}X  \to T_{x,1} X$ and $(\sigma_i^1)_*: T_{x,1}X \to T_{x,2}X$ as
\[
(\sigma_0^0)_*(v) = (v,v,0),
\]
\begin{align*}
(\sigma^1_0)_*(v_1, v_2, \xi)&=(v_1, v_2, \xi, v_2, \xi, 0),& (\sigma^1_1)_*(v_1, v_2, \xi)&=(v_1, v_1, 0, v_2, 0, \xi),
\end{align*}
for $v, v_1, v_2 \in T_x M$ and $\xi \in T^*_x M$. 

At $\sigma^2(x)$, the $2$-form $\omega$ is given in terms of the above identifications by
\[
\omega\Big(\big(v_1, v_2, \xi_1, v_3,\xi_2, \xi_3\big), \big(v_1',v_2',\xi_1', v_3', \xi_2',\xi_3'\big)\Big)=\xi_1'(v_2)-\xi_1(v_2')-\xi_2'(v_3)+\xi_2(v_3'). 
\]
From this, we can compute the pairing $\widetilde{A}_\omega$ between $T_{x,0}X$ and $T_{x,2}X$ to be
\[
\widetilde{A}_\omega\Big(v,\big(v_1, v_2,\xi_1, v_3, \xi_2, \xi_3\big)\Big):=\xi_1(v)-\xi_2(v)
\]
and the pairing $\widetilde{B}_\omega$ on $T_{x,1}$ to be
\[
\widetilde{B}_\omega\Big((v_1, v_2, \xi), (v_1',v_2', \xi')\Big):=\xi'(v_1-v_2)+\xi(v'_1-v'_2). 
\]
One can now see that the kernels of $\widetilde{A}_\omega$ and $\widetilde{B}_\omega$ consist of sums of vectors in the images of $(\sigma_i^q)_*$. Thus the induced pairing $A_\omega$ between $T_{x,0} X$ and $\hat{T}_{x,2}X$ and the induced bilinear form $B_\omega$ on $\hat{T}_{x,1}X$ are both nondegenerate.
\end{proof}

\begin{remark}
The fact that $\hat{T}_{x,2}X$ pairs nondegenerately with $T_{x,0} X = T_x M$ means that $\hat{T}_{x,2} X$ is isomorphic to $T^*_x M$. The isomorphism is explicitly given by composing the map $T^*_x M \to T_{x,2} X$, $\xi \mapsto (0,0,\xi,0,0,0)$ with the quotient map $T_{x,2} X \to \hat{T}_{x,2} X$.

Similarly, $\hat{T}_{x,1} X$ is isomorphic to $T_x M \oplus T^*_x M$, with the isomorphism given by the map $T_x M \oplus T^*_x M \to T_{x,1} X$, $(v, \xi) \mapsto (v, 0, \xi)$. Under this isomorphism, the pairing $B_\omega$ agrees with the standard symmetric pairing on $TM \oplus T^*M$, which is an important part of the Courant algebroid structure.
\end{remark}

\section{Constant symplectic $2$-groupoids}\label{sec:courant}
A \emph{constant symplectic $2$-groupoid} is a symplectic $2$-groupoid $(V_\bullet, \omega)$ where $V_\bullet$ is a simplicial vector space and $\omega \in \Omega^2_\nu(V_2)$ is constant. In this section, we will study constant symplectic $2$-groupoids and obtain a fairly simple description of them. We will later see that there is a correspondence between constant symplectic $2$-groupoids and a certain class of Courant algebroids that we call \emph{constant Courant algebroids}. 

\subsection{Linear $2$-groupoids}\label{subsec:simplicial}
A \emph{linear $2$-groupoid} is a Lie $2$-groupoid $V_\bullet$ such that each $V_q$ is a vector space, and where the face and degeneracy maps are all linear. Linear $2$-groupoids are known to be equivalent, via the Dold-Kan correspondence, to $3$-term chain complexes of vector spaces. Since there are different possible choices of convention and we will require explicit formulas, we will give a brief description of this correspondence.

Suppose that $V_\bullet$ is a linear $2$-groupoid. A $3$-term chain complex $(W_\bullet, \partial)$ is constructed as follows. 

First, we set $W_0 := V_0$. Next, we observe that $f^1_0: V_1 \to W_0$ is a surjection with right inverse $\sigma_0^0$. Thus we have a split short exact sequence
\begin{equation}\label{eq:W_1} \xymatrix{W_1 \ar[r] & V_1 \ar_{f^1_0}[r] & W_0 \ar_{\sigma_0^0}@/_/[l]},
\end{equation}
where $W_1 := \ker f_0^1$, giving us the natural decomposition $V_1 = W_1 \oplus W_0$. Using this decomposition, we define a linear map $\partial_1: W_1 \to W_0$ given by $\partial_1 w_1 = -f_1^1(w_1,0)$. (The sign is chosen to agree with the tangent complex; see Remark \ref{rmk:wtangent}.) We then have the following formulas for the face and degeneracy maps between $V_1$ and $V_0$:
\begin{align*}
f^1_0(w_1, w_0)&=w_0,& f^1_1(w_1, w_0)&= w_0 - \partial_1 w_1,& \sigma^0_0(w_0)&=(0,w_0).
\end{align*}

Now consider the horn space $\Lambda_{2, 2}V$, consisting of pairs $(v_1, v_1') \in V_1 \times V_1$ such that $f_0^1(v_1) = f_0^1(v_1')$. Given such a pair, we may use the decomposition $V_1 = W_1 \oplus W_0$ to write $v_1 = (w_1, w_0)$, $v_1' = (w_1', w_0)$. This allows us to make the identification $\Lambda_{2,2}V = W_1 \oplus W_1 \oplus W_0$, where the pair $(v_1, v_1')$ is identified with $(w_1, w_1', w_0)$.

Let $\Phi_2: \Lambda_{2,2}V \to V_2$ be defined by
\begin{equation}\label{eqn:Phi2}
  \Phi_2(v_1, v_1') = \sigma_0^1 v_1 + \sigma_1^1(v_1' - v_1).
\end{equation}

One can see that $\Phi_2$ is a right inverse of the horn map $\lambda_{2,2}$, so we have a split short exact sequence
\begin{equation}\label{eq:W2} \xymatrix{W_2 \ar[r] & V_2 \ar_{\lambda_{2,2}}[r] & \Lambda_{2,2}V \ar_{\Phi_2}@/_/[l]},
\end{equation}
where $W_2 := \ker \lambda_{2,2}$, giving us a natural decomposition $V_2 = W_2 \oplus \Lambda_{2,2}V = W_2 \oplus W_1 \oplus W_1 \oplus W_0$. Under this decomposition, we have by construction that
\begin{align*}
 f_0^2(w_2, w_1, w_1', w_0) &= (w_1,w_0),\\
 f_1^2(w_2, w_1, w_1', w_0) &= (w_1',w_0). 
\end{align*}
Using \eqref{eqn:Phi2}, we also see that
\begin{align}
 \sigma_0^1(w_1,w_0) &= \Phi_2(w_1,w_1,w_0) = (0,w_1,w_1,w_0), \label{eq:degen1} \\
 \sigma_1^1(w_1,w_0) &= \Phi_2(0,w_1,w_0) = (0,0,w_1,w_0). \label{eq:degen2}
\end{align}
To obtain a formula for $f_2^2$, we first define a map $\partial_2: W_2 \to W_1$, given by \begin{equation}\label{eqn:partial2}
f_2^2(w_2,0,0,0) = (\partial_2 w_2,0).
\end{equation}
Note that \eqref{eqn:partial2} well-defines $\partial_2$, since $f_0^1 f_2^2(w_2,0,0,0) = f_1^1 f_0^2(w_2,0,0,0) = 0$. We then make the following calculation:
\begin{equation}\label{eqn:f2zero}
	\begin{split}
	 f_2^2(0,w_1,w_1',w_0) &= f_2^2 \Phi_2(w_1,w_1',w_0)\\
	 &= f_2^2 \sigma_0^1(w_1,w_0) + f_2^2 \sigma_1^1 (w_1'-w_1,0)\\
	 &= \sigma_0^0 f_1^1 (w_1,w_0) + (w_1' - w_1,0)\\
	 &= (w_1' - w_1, w_0 - \partial_1 w_1).
	\end{split}
\end{equation}
Putting \eqref{eqn:partial2} and \eqref{eqn:f2zero} together, we have
\[ f_2^2(w_2,w_1,w_1',w_0) = (\partial_2 w_2 + w_1' - w_1, w_0 - \partial_1 w_1).\]

We observe that $\partial_1 \partial_2 w_2 = -f_1^1 f_2^2(w_2,0,0,0) = -f_1^1 f_1^2(w_2,0,0,0) = 0$, so
\[
W_2 \xrightarrow{\partial_2} W_1 \xrightarrow{\partial_1} W_0
\]
is a $3$-term chain complex.

Now consider the horn space $\Lambda_{3,3}V$, consisting of triples $(v_2, v_2', v_2'') \in V_2 \times V_2 \times V_2$ such that $f_0^2 v_2 = f_0^2 v_2'$, $f_1^2 v_2 = f_0^2 v_2''$, and $f_1^2 v_2' = f_1^2 v_2''$. In terms of the decomposition $V_2 = W_2 \oplus W_1 \oplus W_1 \oplus W_0$, we may write $v_2 = (w_2, w_1, w_1', w_0)$, $v_2' = (w_2', w_1, w_1'', w_0)$, $v_2'' = (w_2'', w_1', w_1'', w_0)$. This allows us to make the identification $\Lambda_{3,3}V = W_2 \oplus W_2 \oplus W_2 \oplus W_1 \oplus W_1 \oplus W_1 \oplus W_0$, where $(v_2, v_2', v_2'')$ is identified with $(w_2, w_2', w_2'', w_1, w_1', w_1'', w_0)$.

Let $\Phi_3: \Lambda_{3,3}V \to V_3$ be given by 
\[ \Phi_3(v_2, v_2', v_2'') = \sigma_0^2 v_2 + \sigma_1^2(v_2' - v_2) + \sigma_2^2(v_2'' - v_2' + v_2 - \sigma_0^1 f_1^2 v_2).\]
One can see that $\Phi_3$ inverts the horn map $\lambda_{3,3}$, which is assumed to be an isomorphism since $V_\bullet$ is a Lie $2$-groupoid. Implicitly using $\Phi_3$ to identify $V_3$ with $\Lambda_{3,3}V$, we then obtain the following formulas for the face maps:
\begin{align}
 f_0^3(w_2, w_2', w_2'', w_1, w_1', w_1'', w_0) = & v_2 = (w_2,w_1,w_1',w_0), \label{eq:face30}\\
 f_1^3(w_2, w_2', w_2'', w_1, w_1', w_1'', w_0) = & v_2' = (w_2',w_1,w_1'',w_0), \label{eq:face31}\\
 f_2^3(w_2, w_2', w_2'', w_1, w_1', w_1'', w_0) = & v_2'' = (w_2'',w_1',w_1'',w_0),\label{eq:face32}\\ 
 \begin{split}
 f_3^3(w_2, w_2', w_2'', w_1, w_1', w_1'', w_0) = & (w_2'' - w_2' + w_2,\partial_2 w_2 + w_1' - w_1,\\
 &\partial_2 w_2' + w_1'' - w_1, w_0 - \partial_1 w_1).\label{eq:face33}
\end{split}
\end{align}

We have seen that, given a linear $2$-groupoid $V_\bullet$, we can obtain a $3$-term chain complex $(W_\bullet, \partial)$. Conversely, given a $3$-term chain complex, one can construct a linear $2$-groupoid by setting $V_0 = W_0$, $V_1 = W_1 \oplus W_0$, etc., with the face and degeneracy maps given in low degrees by the above formulas. As $V_\bullet$ is a Lie $2$-groupoid, the higher simplicial maps are completely determined by the data in low degrees.

\begin{remark}\label{rmk:wtangent}
	From the short exact sequences \eqref{eq:W_1} and \eqref{eq:W2}, we can see that the $3$-term chain complex $(W_\bullet, \partial)$ is naturally isomorphic to the tangent complex of $V_\bullet$. As a result, we obtain an alternative description of the tangent complex of a Lie $2$-groupoid $X_\bullet$, where $\hat{T}_{x,2} X = \ker (\lambda_2^2)_* \subseteq T_{x,2} X$, $\hat{T}_{x,1} X = \ker (\lambda_1^1)_* = \ker (f_0^1)_* \subseteq T_{x,1} X$, and where the boundary map is $\partial_q = (-1)^q (f_q^q)_*$.
\end{remark}

\subsection{Constant multiplicative $2$-forms}\label{subsec:multiplicative}
Let $V_\bullet$ be a linear $2$-groupoid. In this subsection, we will obtain a description of constant multiplicative $2$-forms $\omega \in \Omega_\nu^2(V_2)$ in terms of data on the associated $3$-term complex $(W_\bullet, \partial)$.

Recall from Section \ref{subsec:simplicial} that $V_2$ can be naturally decomposed as $W_2\oplus W_1\oplus W_1\oplus W_0$. With respect to this decomposition, we may write any constant $\omega \in \Omega^2(V_2)$ as a sum of bilinear forms $C_{11}\in \wedge^2 W_2^*$, $C_{12}, C_{13} \in W_2^*\otimes W_1^*$, $C_{14}\in W_2^*\otimes W_0^*$, 
$C_{22}, C_{33} \in \wedge^2 W_1^*$, $C_{23}\in W_1^*\otimes W_1^*$, $C_{24}, C_{34} \in W_1^*\otimes W_0^*$, and $C_{44}\in \wedge^2 W_0^*$. Then we can write $\omega$ in the form of a block matrix
\[
\omega=\begin{bmatrix}
C_{11}&C_{12}&C_{13}&C_{14}\\
C_{21}&C_{22}&C_{23}&C_{24}\\
C_{31}&C_{32}&C_{33}&C_{34}\\
C_{41}&C_{42}&C_{43}&C_{44}
\end{bmatrix},
\]
where $C_{ij}(w,w') = -C_{ji}(w',w)$.

Now suppose that $\omega$ is normalized. From \eqref{eq:degen2}, we immediately see that $C_{33}$, $C_{34}$, and $C_{44}$ vanish. From \eqref{eq:degen1}, we then see that $C_{24} = 0$, since
\[ 0 = \omega( (0,w_1,w_1,w_0),(0,0,0,w_0')) = C_{24}(w_1, w_0')\]
for all $w_1 \in W_1$ and $w_0,w_0' \in W_0$, and that
\[ 0 = \omega( (0,w_1,w_1,0),(0,w_1',w_1',0)) = C_{22}(w_1, w_1') + C_{23}(w_1, w_1') + C_{32}(w_1,w_1')\]
for all $w_1, w_1' \in W_1$. Thus, $\omega$ is of the form
\begin{equation}\label{eq:omeganorm}
\omega=\begin{bmatrix}
C_{11}&C_{12}&C_{13}&C_{14}\\
C_{21}&C_{22}&C_{23}&0\\
C_{31}&C_{32}&0&0\\
C_{41}&0&0&0
\end{bmatrix},
\end{equation}
where
\begin{equation}\label{eq:normc}
C_{22}+C_{23}+C_{32}=0.
\end{equation}
Conversely, it is straightforward to check that, if $\omega$ is of the form \eqref{eq:omeganorm} and satisfies \eqref{eq:normc}, then $\omega$ is normalized.

Now suppose that $\omega$ is normalized and multiplicative. In the following series of lemmas, we find further relations that hold between the bilinear forms $C_{ij}$. The proofs utilize the decomposition $V_3 = W_2 \oplus W_2 \oplus W_2 \oplus W_1 \oplus W_1 \oplus W_1 \oplus W_0$ constructed in Section \ref{subsec:simplicial}, as well as the formulas \eqref{eq:face30}--\eqref{eq:face33} for the face maps $f_i^3$.

\begin{lemma}\label{lem:13-23} $C_{13}(w_2, w_1) = C_{32}(w_1, \partial w_2)$ for all $w_2\in W_2$ and $w_1\in W_1$.
\end{lemma}
\begin{proof}
Let $u,v \in V_3$ be defined as $u=(w_2,0,0, 0,0,0,0)$ and $v=(0,0,0,0,0,w_1, 0)$. Then 
\begin{equation*}
\begin{split}
0 = \delta \omega(u,v)&= \sum_{i=0}^3 (-1)^i \omega(f_i^3 u, f_i^3 v) \\
 &= -\omega\big( (w_2,\partial w_2, 0,0 ), (0,0,w_1,0 )\big)\\
 &= -C_{13}(w_2,w_1)-C_{23}(\partial w_2, w_1)\\
 &= -C_{13}(w_2,w_1) + C_{32}(w_1, \partial w_2).\qedhere
 \end{split}
 \end{equation*}
\end{proof}

\begin{lemma}\label{lem:12-23} $C_{12}(w_2,w_1)=C_{32}(\partial w_2, w_1)$ for all $w_2\in W_2, w_1\in W_1$. 
\end{lemma}
\begin{proof}
Let $u,v \in V_3$ be defined as $u=(0,w_2,0,0,0,0,0)$ and $v=(0,0,0,0,w_1,0,0)$. Then
\begin{equation*}
\begin{split}
0 = \delta \omega(u,v) &= -\omega \big( (-w_2,0,\partial w_2,0), (0,w_1,0,0) \big)\\
&= C_{12}(w_2, w_1) - C_{32}(\partial w_2, w_1). \qedhere
\end{split}
\end{equation*}
\end{proof}

\begin{lemma}\label{lem:11-12} $C_{11}(w_2,w_2')=-C_{32}(\partial w_2,\partial w_2')$ for all $w_2, w_2' \in W_2$. 
\end{lemma}
\begin{proof}
Let $u,v \in V_3$ be defined as $u=(0,0,w_2,0,0,0,0)$ and $v=(w_2',0,0,0,0,0,0)$. Then
\begin{equation*}
\begin{split}
0 = \delta \omega(u,v) &= -\omega \big( (w_2,0,0,0), (w_2',\partial w_2',0,0) \big)\\
&= -C_{11}(w_2, w_2') -C_{12}(w_2, \partial w_2').
\end{split}
\end{equation*}
By Lemma \ref{lem:12-23}, we see that $C_{11}(w_2,w_2') = -C_{12}(w_2, \partial w_2') = -C_{32}(\partial w_2, \partial w_2')$.
\end{proof}

Lemmas \ref{lem:13-23}--\ref{lem:11-12}, together with \eqref{eq:normc}, show that $\omega$ is completely determined by $C_{41}$ and $C_{32}$. 

\begin{thm}\label{thm:multiplicative} Let $V_\bullet$ be a linear $2$-groupoid with associated $3$-term chain complex $(W_\bullet, \partial)$. There is a one-to-one correspondence between constant multiplicative $2$-forms $\omega \in \Omega^2_\nu(V_2)$ and pairs $(C_{41},C_{32})$, where $C_{41}$ is a bilinear pairing of $W_0$ with $W_2$ and $C_{32}$ is a bilinear form on $W_1$, such that
	\begin{equation}\label{eq:c1423} C_{41}(\partial w_1, w_2) = C_{32}(\partial w_2, w_1) + C_{32}(w_1, \partial w_2)
	\end{equation}
	for all $w_2 \in W_2$ and $w_1 \in W_1$.

\end{thm}
\begin{proof}
	In the discussion above, we have already seen how to obtain $C_{41}$ and $C_{32}$ from $\omega$. To see that \eqref{eq:c1423} holds, let $u,v \in V_3$ be defined as $u = (0,0,0,w_1,0,0,0)$ and $v = (0,0, w_2, 0,0,0,0)$. Then
	\begin{equation}
		\begin{split}
		0 = \delta \omega(u,v) &= -\omega \big( (0,-w_1,-w_1,-\partial w_1), (w_2,0,0,0) \big)\\
		&= -C_{12}(w_2,w_1) - C_{13}(w_2,w_1) + C_{41}(\partial w_1, w_2).
		\end{split}
	\end{equation}
	Equation \eqref{eq:c1423} then follows from Lemmas \ref{lem:12-23} and \ref{lem:11-12}.
	
	In the other direction, given $C_{41}$ and $C_{32}$ satisfying \eqref{eq:c1423}, we can construct $\omega$ of the form \eqref{eq:omeganorm}, with the other components given by \eqref{eq:normc} and Lemmas \ref{lem:13-23}--\ref{lem:11-12}. The skew-symmetry of $C_{22}$ follows from \eqref{eq:normc}, and the skew-symmetry of $C_{11}$ follows from Lemma \ref{lem:11-12} and \eqref{eq:c1423}. We have already observed that such an $\omega$ will be normalized. It is long but straightforward to check that $\omega$ is multiplicative.
\end{proof}

\subsection{Simplicial nondegeneracy}
\label{subsec:nondeg}
Let $V_\bullet$ be a linear $2$-groupoid equipped with a constant $2$-form $\omega \in \Omega^2_\nu(V_2)$. We will now describe the pairings $A_\omega$ and $B_\omega$ from Section \ref{subsec:lie2} in terms of the components in \eqref{eq:omeganorm}.

Recall from Remark \ref{rmk:wtangent} that the $3$-term complex $(W, \partial)$ associated to $V_\bullet$ is isomorphic to the tangent complex of $V_\bullet$. Thus we can view $A_\omega$ and $B_\omega$ as the restrictions of \eqref{eqn:tildeA} and \eqref{eqn:tildeB}, respectively, to $W_i$. Therefore
\begin{align*}
	A_\omega(w_0,w_2) &= \omega \big( (0,0,0,w_0), (w_2,0,0,0) \big) = C_{41}(w_0,w_2),\\
	\begin{split}
		B_\omega(w_1,w_1') &= \omega \big( (0,0,w_1,0),(0,w_1',w_1',0) \big) + \omega \big( (0,0,w_1',0),(0,w_1,w_1,0) \big) \\
		&= C_{32}(w_1,w_1') + C_{32}(w_1',w_1),
	\end{split}
\end{align*}
for $w_0 \in W_0$, $w_1,w_1' \in W_1$, and $w_2 \in W_2$. The following result is immediate.
\begin{prop}
\label{prop:nondegeneracy} $\omega$ is simplicially nondegenerate if and only if $C_{41}$ and the symmetric part of $C_{32}$ are both nondegenerate. 
\end{prop}

\begin{remark}\label{rmk:symmetricnondegen}
	Proposition \ref{prop:nondegeneracy} allows us to clearly see the difference between simplicial nondegeneracy and the ordinary notion of nondegeneracy for $2$-forms. From \eqref{eq:omeganorm}, it is clear that $\omega$ is nondegenerate in the ordinary sense if and only if $C_{41}$ and $C_{32}$ are both nondegenerate. Thus the difference is that ordinary nondegeneracy considers $C_{32}$ in its entirety, whereas simplicial nondegeneracy only considers the symmetric part of $C_{32}$. In the case where $C_{32}$ is symmetric, the two notions agree; however, since $C_{32}$ need not be symmetric in general, it is easy to find examples of $2$-forms that are simplicially nondegenerate but not nondegenerate, and vice versa. 
\end{remark}

\subsection{A minimal description of constant symplectic $2$-groupoids}

Putting Theorem \ref{thm:multiplicative} and Proposition \ref{prop:nondegeneracy} together, we see that there is a one-to-one correspondence between constant symplectic $2$-groupoids and $3$-term chain complexes $(W_\bullet, \partial)$ equipped with a nondegenerate bilinear pairing $C_{41}$ of $W_0$ with $W_2$ and a bilinear form $C_{32}$ on $W_1$ whose symmetric part is nondegenerate, satisfying \eqref{eq:c1423}. Using the nondegeneracy of the pairings, we can further simplify the description.

\begin{thm}\label{thm:chaincorrespondence}
 There is a one-to-one correspondence between constant symplectic $2$-groupoids and tuples $(W_1, W_0, \langle \cdot, \cdot \rangle, \partial, r)$, where 
 \begin{itemize}
  \item $W_1$ and $W_0$ are vector spaces, 
  \item $\langle \cdot, \cdot \rangle$ is a nondegenerate symmetric bilinear form on $W_1$, 
  \item $\partial: W_1 \to W_0$ is a linear map such that the image of $\partial^*$ in $W_1^* \cong W_1$ is isotropic, and
  \item $r$ is an element of $\wedge^2 W_1^*$.
 \end{itemize}
\end{thm}
\begin{proof}
Given the data $(W_1, W_0, \langle \cdot, \cdot \rangle, \partial, r)$, the corresponding $3$-term chain complex is
\begin{equation}\label{eqn:3constant}
  W_0^* \xrightarrow{\partial^*} W_1 \xrightarrow{\partial} W_0,
\end{equation}
where we are implicitly using $\langle \cdot, \cdot \rangle$ to identify $W_1$ with $W_1^*$. The equation $\partial \circ \partial^* = 0$ is equivalent to the requirement that the image of $\partial^*$ be isotropic. We take $C_{41}$ to be the canonical pairing of $W_0^*$ with $W_0$, and we set $C_{32} := \frac{1}{2}\langle \cdot, \cdot \rangle + r$. The equation \eqref{eq:c1423} automatically holds. One can easily check that this gives a one-to-one correspondence.
\end{proof}

\subsection{Equivalences}
\label{subsec:equiv}
In this subsection, we will describe equivalences between constant symplectic $2$-groupoids in terms of the description given in Theorem \ref{thm:chaincorrespondence}.

From Definition \ref{dfn:equivalence} and Remark \ref{rmk:equivalence}, we can see that any equivalence between constant symplectic $2$-groupoid structures on a linear $2$-groupoid $V_\bullet$ is given by a $2$-form $\alpha \in \Omega^2_\nu(V_1)$ satisfying \eqref{eqn:Adelta} and \eqref{eqn:Bdelta}. Since the simplicial coboundary operator $\delta$ is linear, we may assume without loss of generality that $\alpha$ is constant.

Let $\alpha$ be a constant normalized $2$-form on $V_1$. In terms of the decomposition $V_1 = W_1 \oplus W_0$, we can write $\alpha$ in block form as
\[ \alpha = \begin{bmatrix} B_{11} & B_{12} \\ B_{21} & 0 \end{bmatrix},\]
where $B_{12}(w_1, w_0) = -B_{21}(w_0, w_1)$. The vanishing of the lower right block is a consequence of the assumption that $\alpha$ is normalized. A straightforward calculation then shows that \eqref{eqn:Adelta} and \eqref{eqn:Bdelta} reduce in this case to the conditions
\begin{gather}
B_{21}(w_0, \partial w_2) = 0, \label{eqn:alpha1}\\
B_{21}(\partial w_1, w_1') + B_{21}(\partial w_1', w_1) = 0, \label{eqn:alpha2}
\end{gather}
for all $w_2 \in W_2$, $w_1, w_1' \in W_0$, and $w_0 \in W_0$.

Since $\delta \alpha$ is obviously multiplicative, it has block form
\[ \delta \alpha = \begin{bmatrix} A_{11} & A_{12} & A_{13} & A_{14} \\ A_{21} & A_{22} & A_{23} & 0 \\ A_{31} & A_{32} & 0 & 0 \\ A_{41} & 0 & 0 & 0 \end{bmatrix}.\]
By Theorem \ref{thm:multiplicative}, $\delta \alpha$ is completely determined by $A_{41}$ and $A_{32}$. We calculate
\begin{equation*}
\begin{split}
A_{41}(w_0, w_2) &= \delta \alpha \big( (0,0,0,w_0), (w_2,0,0,0) \big)\\
&= \alpha \big((0,w_2),(\partial w_2,0) \big) \\
&= B_{21}(w_0,\partial w_2),
\end{split}
\end{equation*}
\begin{equation*}
\begin{split}
A_{32}(w_1, w_1') &= \delta \alpha \big( (0,0,w_1,0), (0,w_1',0,0) \big) \\
&= \alpha \big( (w_1, 0), (-w_1', -\partial w_1') \big) \\
&= -B_{11}(w_1,w_1') - B_{12}(w_1, \partial w_1').
\end{split}
\end{equation*}
From this, we see that equation \eqref{eqn:alpha1} holds if and only if $A_{41} = 0$, and that \eqref{eqn:alpha2} holds if and only if $A_{32}$ is skew-symmetric. Furthermore, we observe that there is no restriction on $B_{11}$, so every skew-symmetric pairing on $W_1$ appears as $A_{32}$ for some choice of $\alpha$ satisfying $\eqref{eqn:alpha1}$ and $\eqref{eqn:alpha2}$.

It is clear from the above discussion that, in terms of the data of Theorem \ref{thm:chaincorrespondence}, equivalences between constant symplectic groupoids act transitively on $r$ and do not affect any of the other data. Thus we have the following result.
\begin{thm}\label{thm:constantequivalence}
 There is a one-to-one correspondence between equivalence classes of constant symplectic $2$-groupoids and tuples $(W_1, W_0, \langle \cdot, \cdot \rangle, \partial)$, where 
 \begin{itemize}
  \item $W_1$ and $W_0$ are vector spaces, 
  \item $\langle \cdot, \cdot \rangle$ is a nondegenerate symmetric bilinear form on $W_1$, and
  \item $\partial: W_1 \to W_0$ is a linear map such that the image of $\partial^*$ in $W_1^* \cong W_1$ is isotropic.
 \end{itemize}
\end{thm}

\begin{remark}\label{rmk:symmetric}
In each equivalence class of constant symplectic $2$-groupoids, there is exactly one representative for which $C_{32}$ is symmetric (or equivalently, in terms of the data of Theorem \ref{thm:chaincorrespondence}, where $r=0$). In this case, we will say that the constant symplectic $2$-groupoid is \emph{symmetric}. From Remark \ref{rmk:symmetricnondegen}, we can see that if $(V_\bullet, \omega)$ is a symmetric constant symplectic $2$-groupoid, then $\omega \in \Omega^2(V_2)$ is genuinely nondegenerate, and therefore $V_2$ is genuinely symplectic.
\end{remark}

\section{Integration of constant Courant algebroids}\label{sec:constcourant}
In this section, we will describe a simple class of Courant algebroids that we call \emph{constant Courant algebroids}. We will see that constant Courant algebroids are in correspondence with equivalence classes of constant symplectic $2$-groupoids.

\subsection{Constant Courant algebroids}
We first recall the definition of Courant algebroid.
\begin{definition}\label{def:courant}
A \emph{Courant algebroid} is a vector bundle $E \to M$ equipped with a nondegenerate symmetric bilinear form $\langle \cdot, \cdot \rangle$, a bundle map $\rho: E \to TM$ (called the \emph{anchor}), and a bracket $\cbrack{\cdot}{\cdot}$
(called the \emph{Courant bracket}) on $\Gamma(E)$ such that
\begin{enumerate}
\item $\cbrack{e_1}{fe_2} = \rho(e_1)(f) e_2 + f\cbrack{e_1}{e_2}$,
\item $\rho(e_1)(\langle e_2, e_3 \rangle) = \langle \cbrack{e_1}{e_2}, e_3 \rangle + \langle e_2, \cbrack{e_1}{e_3} \rangle$,
\item $\cbrack{\cbrack{e_1}{e_2}}{e_3} = \cbrack{e_1}{\cbrack{e_2}{e_3}} - \cbrack{e_2}{\cbrack{e_1}{e_3}}$,
\item $\cbrack{e_1}{e_2} + \cbrack{e_2}{e_1} = \cD \langle e_1, e_2 \rangle$,
\end{enumerate}
for all $f \in C^\infty(M)$ and $e_i \in \Gamma(E)$, where $\cD : C^\infty(M) \to \Gamma(E)$ is defined by
\begin{equation*}
\langle \cD f, e \rangle = \rho(e)(f).
\end{equation*}
\end{definition}
A well-known but important consequence of Definition \ref{def:courant} is that, for any Courant algebroid $E \to M$, the sequence
\begin{equation}\label{eqn:3courant}
  T^*M \xrightarrow{\rho^*} E^* \cong E \xrightarrow{\rho} TM 
\end{equation}
is a $3$-term chain complex of vector bundles over $M$. One should expect that a symplectic $2$-groupoid integrating $E$ should be such that its tangent complex is isomorphic (or at least quasi-isomorphic) to \eqref{eqn:3courant}.

\begin{definition}\label{def:constantcourant}
	A \emph{constant Courant algebroid} is a Courant algebroid of the form $W_1 \times W_0 \to W_0$, where $W_0$ and $W_1$ are vector spaces, such that
	\begin{enumerate}
	 \item the pairing $\langle \cdot, \cdot \rangle$ is independent of the basepoint in $W_0$,
	 \item the anchor $\rho: W_1 \times W_0 \to TW_0 = W_0 \times W_0$ is independent of the basepoint, and
	 \item $\cbrack{w_1}{w_1'} = 0$ for all $w_1, w_1' \in W_1$, viewed as constant sections.
	\end{enumerate}
\end{definition}
Let $E = W_1 \times W_0 \to W_0$ be a constant Courant algebroid. The space of sections $\Gamma(E)$ can be naturally identified with $C^\infty(W_0) \otimes W_1$. Any bilinear pairing on $\Gamma(E)$ is completely determined by its restriction to constant sections, which gives a symmetric map $W_1 \otimes W_1 \to C^\infty(W_0)$. The requirement that the pairing be independent of basepoint implies the image of this map consists of constant functions; in other words, the pairing is given by a nondegenerate symmetric bilinear form on $W_1$. Similarly, the requirement that the anchor be independent of basepoint implies that it is given by 
\begin{equation}\label{eqn:anchor}
\rho(w_1, w_0) = (\partial w_1, w_0)
\end{equation}
for some linear map $\partial: W_1 \to W_0$. 

The fact that $\rho \circ \rho^* = 0$ (see \eqref{eqn:3courant}) implies that $\partial \circ \partial^* = 0$ or, equivalently, the image of $\partial^*$ in $W_1^* \cong W_1$ is isotropic.

The vanishing of the bracket of constant sections, together with axioms (1) and (4) in Definition \ref{def:courant}, imply that the bracket is completely determined by the pairing and anchor. Specifically,
\begin{equation}\label{eqn:bracket}
\cbrack{fw_1}{gw_1'} = f \rho(w_1) (g) w_1' - g \rho(w_1')(f) w_1 + g \langle w_1, w_1' \rangle \cD f
\end{equation}
for $f,g \in C^\infty(W_0)$ and $w_1, w_1' \in W_1$.

\begin{thm}\label{thm:correspondence}
There is a one-to-one correspondence between constant Courant algebroids and equivalence classes of constant symplectic $2$-groupoids.
\end{thm}
\begin{proof}
Given a constant Courant algebroid, we have seen how to obtain the data $(W_1, W_0, \langle \cdot, \cdot \rangle, \partial)$ for an equivalence class of  constant symplectic $2$-groupoids, as described in Theorem \ref{thm:constantequivalence}. On the other hand, given the data $(W_1, W_0, \langle \cdot, \cdot \rangle, \partial)$, we can construct the constant Courant algebroid $E = W_1 \times W_0 \to W_0$, where the pairing on $\Gamma(E)$  agrees on constant sections with the pairing on $W_0$, and the anchor and bracket are given by \eqref{eqn:anchor} and \eqref{eqn:bracket}. It is a long but straightforward check that the conditions of Definition \ref{def:courant} hold. (We note that, alternatively, the Courant algebroid axioms can be checked quickly using the supergeometric formulation of \cite{dmitry:graded} in local coordinates.)
\end{proof}

\section{Integration of constant Dirac structures}\label{sec:dirac}

\subsection{Linear sub-$2$-groupoids}

Let $V_\bullet$ be a linear $2$-groupoid. A \emph{linear sub-$2$-groupoid} of $V_\bullet$ is a simplicial subspace $L_\bullet \subseteq V_\bullet$ that is also a $2$-groupoid. 

Let $L_\bullet \subseteq V_\bullet$ be a linear sub-$2$-groupoid. Set $U_0 := L_0$, $U_1 := \ker f_0^1|_{L_1}$, and $U_2 := \ker \lambda_{2,2}|_{L_2}$. By comparing with the constructions of Section \ref{subsec:simplicial}, we can see that $U_\bullet$ is a subcomplex of $(W_\bullet, \partial)$. Conversely, given a subcomplex $U_\bullet \subseteq (W_\bullet, \partial)$, we can form a linear sub-$2$-groupoid $L_\bullet$, where $L_0 = U_0$, $L_1 = U_1 \oplus U_0$, and $L_2 = U_2 \oplus U_1 \oplus U_1 \oplus U_0$. This gives us the following result.

\begin{prop}\label{prop:sub2gpd}
There is a one-to-one correspondence between linear sub-$2$-groupoids $L_\bullet \subseteq V_\bullet$ and subcomplexes $U_\bullet \subseteq (W_\bullet, \partial)$.
\end{prop}

\subsection{Linear Lagrangian sub-$2$-groupoids}

Now suppose that $(V_\bullet, \omega)$ is the symmetric constant symplectic $2$-groupoid corresponding (via Theorem \ref{thm:constantequivalence}) to the data $(W_1, W_0, \langle \cdot, \cdot \rangle, \partial)$. In this case, $V_2 = W_0^* \oplus W_1 \oplus W_1 \oplus W_0$, and the block form \eqref{eq:omeganorm} is such that $C_{41}$ is the natural pairing of $W_0^*$ with $W_0$ and $C_{32}$ is the pairing $\langle \cdot, \cdot \rangle$. 
It follows that $\omega$ is a genuine symplectic form on $V_2$ (see Remark \ref{rmk:symmetric}).

If $U_\bullet$ is a subcomplex of $(W_\bullet, \partial)$, then we can consider the space $L_2^\omega \subseteq V_2$ that is symplectic orthogonal to $L_2 = U_2 \oplus U_1 \oplus U_1 \oplus U_0$. The following gives a description of $L_2^\omega$ in terms of $U_\bullet$.
\begin{lemma}\label{lemma:l2omega}
 $L_2^\omega = \Ann(U_0) \oplus U_1^\perp \oplus U_1^\perp \oplus \Ann(U_2)$.
\end{lemma}
\begin{proof}
	For any $w_2 \in W_0^*$, $w_1, w_1' \in W_1$, $w_0 \in W_0$, and $(u_2, u_1, u_1', u_0) \in L_2$, we use \eqref{eq:normc} and Lemmas \ref{lem:13-23}--\ref{lem:11-12} together with the symmetry of $C_{32}$ to derive the formulas
\begin{equation}\label{eqn:omegau4}
\begin{split}
\omega((w_2,0,0,0),(u_2,u_1,u_1',u_0)) &= C_{11}(w_2,u_2) + C_{12}(w_2,u_1) + C_{13}(w_2,u_1') + C_{14}(w_2,u_0) \\
&= C_{32}(\partial^* w_2, u_1 + u_1' - \partial^* u_2) - C_{41}(u_0, w_2) \\
&= \langle \partial^* w_2, u_1 + u_1' - \partial^* u_2 \rangle - w_2(u_0) \\
&= w_2(\partial u_1 + \partial u_1' - u_0),
\end{split}
\end{equation}
\begin{equation}\label{eqn:omegau6}
\begin{split}
\omega((0,w_1,0,0),(u_2,u_1,u_1',u_0)) &= C_{21}(w_1,u_2) + C_{23}(w_1, u_1') \\
&= -C_{32}(\partial^* u_2, w_1) - C_{32}(u_1', w_1) \\
&= -\langle w_1, u_1' + \partial^* u_2 \rangle,
\end{split}
\end{equation}
\begin{equation}\label{eqn:omegau5}
\begin{split}
\omega((0,0,w_1',0),(u_2,u_1,u_1',u_0)) &= C_{31}(w_1',u_2) + C_{32}(w_1', u_1) \\
&= -C_{32}(w_1', \partial^* u_2) + C_{32}(w_1', u_1) \\
&= \langle w_1', u_1 - \partial^* u_2 \rangle,
\end{split}
\end{equation}
\begin{equation}\label{eqn:omegau3}
\begin{split}
\omega((0,0,0,w_0),(u_2,u_1,u_1',u_0)) &= C_{41}(w_0,u_2)\\
&= u_2(w_0).
\end{split}
\end{equation}
From \eqref{eqn:omegau4}--\eqref{eqn:omegau3} it is immediate that $\Ann(U_0) \oplus U_1^\perp \oplus U_1^\perp \oplus \Ann(U_2) \subseteq L_2^\omega$. 

Conversely, if $(w_2, w_1, w_1', w_0) \in V_2$ is in $L_2^\omega$, then from \eqref{eqn:omegau4}--\eqref{eqn:omegau3} we have \[ 0 = \omega((w_2,w_1,w_1', w_0), (0,0,0,u_0)) = -w_2(u_0)\]
for all $u_0 \in U_0$, so it follows that $w_2$ is in $\Ann(U_0)$. Similarly,
\[ 0 = \omega((w_2,w_1,w_1', w_0), (0,0,u_1',0)) = w_2(\partial u_1') - \langle w_1, u_1' \rangle = - \langle w_1, u_1' \rangle\]
and
\[ 0 = \omega((w_2,w_1,w_1', w_0), (0,u_1,0,0)) = w_2(\partial u_1) + \langle w_1', u_1 \rangle = \langle w_1', u_1 \rangle\]
for all $u_1, u_1' \in U_1$, so it follows that $w_1$ and $w_1'$ are in $U_1^\perp$. Finally,
\[ 0 = \omega((w_2,w_1,w_1', w_0), (u_2,0,0,0)) = -\langle w_1 + w_1', \partial^* u_2 \rangle + u_2(w_0) = u_2(w_0)\]
for all $u_2 \in U_2$, so it follows that $w_0$ is in $\Ann(U_2)$.
\end{proof}

The following is an immediate consequence of Lemma \ref{lemma:l2omega}.
\begin{cor}\label{cor:lagrangian}
\begin{enumerate}
\item $L_2$ is isotropic if and only if $U_0 \subseteq W_0$ and $U_2 \subseteq W_0^*$ pair to zero and $U_1 \subseteq W_1$ is isotropic.
\item $L_2$ is coisotropic if and only if the annihilator of $U_2$ is contained in $U_0$, the annihilator of $U_0$ is contained in $U_2$, and $U_1^\perp \subseteq U_1$.
\item $L_2$ is Lagrangian if and only if $U_2 = \Ann(U_0)$ and $U_1^\perp = U_1$.
\end{enumerate}
\end{cor}
We can now obtain a description of linear Lagrangian sub-$2$-groupoids in terms of the data of Theorem \ref{thm:constantequivalence}.
\begin{thm}\label{thm:lagrangian}
Suppose that $(V_\bullet, \omega)$ is the symmetric constant symplectic $2$-groupoid corresponding to the data $(W_1, W_0, \langle \cdot, \cdot \rangle, \partial)$. There is a one-to-one correspondence between linear Lagrangian sub-$2$-groupoids $L_\bullet \subseteq V_\bullet$ and pairs $(U_1, U_0)$, where $U_i \subseteq W_i$ for $i=0,1$ are such that $U_1^\perp = U_1$ and $\partial(U_1) \subseteq U_0$.
\end{thm}
\begin{proof}
It is immediate from Proposition \ref{prop:sub2gpd} and Corollary \ref{cor:lagrangian} that, if $L_\bullet$ is Lagrangian, then the corresponding subspaces $U_i \subseteq W_i$ satisfy $U_1^\perp = U_1$ and $\partial(U_1) \subseteq U_0$. In the other direction, we set $U_2 = \Ann(U_0)$. For any $u_2 \in U_2$ and $u_1 \in U_1$, we have
\[ \langle \partial^* u_2, u_1 \rangle = u_2(\partial u_1) = 0\]
for all $u_1 \in U_1$, so $\partial^* u_2$ is in $U_1^\perp = U_1$. Therefore, $U_\bullet$ is a subcomplex, and by Corollary \ref{cor:lagrangian} the corresponding $L_2$ is Lagrangian.
\end{proof}

\subsection{Constant Dirac structures}

A \emph{Dirac structure} in a Courant algebroid $E \to M$ is a subbundle $D \to M$ such that $D^\perp= D$ and $\Gamma(D)$ is closed under the Courant bracket\footnote{We note that we are using the condition $D^\perp= D$ in place of the usual requirement that $D$ be maximally isotropic. In most cases of interest, $E$ has signature $(n,n)$, and the two conditions are equivalent. However, if $E$ does not have signature $(n,n)$, then, according to the definition we are using, there do not exist any Dirac structures in $E$.}. The restriction of the Courant bracket $\cbrack{\cdot}{\cdot}$ to any Dirac structure $D$ is a Lie bracket, making $D \to M$ into a Lie algebroid.

Let $W_1 \times W_0 \to W_0$ be a constant Courant algebroid. We will restrict our attention to \emph{constant Dirac structures}, i.e.\ Dirac structures of the form $U_1 \times W_0$, where $U_1$ is a subspace of $W_1$.

\begin{lemma}\label{lemma:constantdirac}
Let $U_1$ be a subspace of $W_1$. Then $U_1 \times W_0$ is a constant Dirac structure if and only if $U_1^\perp = U_1$.
\end{lemma}
\begin{proof}
Since the pairing on $W_1 \times W_0 \to W_0$ is constant, it is immediate that $U_1^\perp = U_1$ if and only if $(U_1 \times W_0)^\perp = U_1 \times W_0$. The nontrivial part of the lemma is the observation that, in this case, closure under the Courant bracket is a consequence of the isotropic property. To see this, suppose that $U_1 \subseteq W_1$ is isotropic. Then, from \eqref{eqn:bracket}, we have
\[ \cbrack{fu}{gu'} = f \rho(u)(g)u' - g \rho(u')(f)u + g \langle u, u' \rangle \cD f\]
for $f,g \in C^\infty(W_0)$ and $u,u' \in U_1$. The last term vanishes, and the other two terms are clearly in $\Gamma(U_1 \times W_0)$.
\end{proof}

Given $U_1 \subseteq W_1$ such that $U_1^\perp = U_1$, we may apply Theorem \ref{thm:lagrangian} to the pair $(U_1,W_0)$ to obtain a linear Lagrangian sub-$2$-groupoid $L_\bullet \subseteq V_\bullet$ that is \emph{wide}, in the sense that $L_0 = V_0$. It is clear from the correspondence of Theorem \ref{thm:lagrangian} that every wide linear Lagrangian sub-$2$-groupoid arises in this manner. Using Lemma \ref{lemma:constantdirac}, we then obtain the following result.

\begin{thm}\label{thm:diracintegrate}
Suppose that $(V_\bullet, \omega)$ is the symmetric constant symplectic $2$-groupoid corresponding to the data $(W_1, W_0, \langle \cdot, \cdot \rangle, \partial)$, and let $W_1 \times W_0 \to W_0$ be the corresponding constant Courant algebroid. There is a one-to-one correspondence between constant Dirac structures $U_1 \times W_0 \subseteq W_1 \times W_0$ and wide linear Lagrangian sub-$2$-groupoids $L_\bullet \subseteq V_\bullet$.
\end{thm}

\begin{remarks}
\begin{enumerate}
\item From part (3) of Corollary \ref{cor:lagrangian}, we can see that a linear Lagrangian sub-$2$-groupoid corresponding to the subcomplex $U_\bullet \subseteq (W_\bullet, \partial)$ is wide if and only if $U_2 = \{0\}$. Therefore, a wide linear Lagrangian sub-$2$-groupoid is actually a $1$-groupoid $U_1 \oplus W_0 \arrows W_0$. It is straightforward to check that this is the Lie groupoid that integrates the Lie algebroid $U_1 \times W_0 \to W_0$.
\item We stress that there are many complications involved in extending the result of Theorem \ref{thm:diracintegrate} to the nonlinear situation. In particular, we expect that a weaker definition of Lagrangian sub-$2$-groupoid, using some of the ideas of derived symplectic geometry \cite{ptvv:derived}, will be required.
\end{enumerate}
\end{remarks}

\bibliographystyle{amsplain}
\bibliography{courantbib-revised}

\end{document}